\documentclass[12pt]{amsart}

\usepackage{amsaddr}

\author[J. Spielmann]{J\'{e}r\^{o}me Spielmann*}
\thanks{*: I would like to thank Lioudmila Vostrikova for helpful discussions and for taking the time to read through the first draft of the manuscript. I would also like to acknowledge financial support from the D\'{e}fiMaths project of the "F\'{e}d\'{e}ration de Recherche Math\'{e}matiques des Pays de Loire", form the PANORisk project of the R\'{e}gion Pays de la Loire and from the French government's "Investissements d'Avenir" program ANR-11-LABX-0020-01.}
\address{LAREMA, Universit\'{e} d'Angers, 2 Bd. Lavoisier, 49045 Angers}
\email{jerome.spielmann@univ-angers.fr}
\title[Classification of the Bounds on the Probability of Ruin]{Classification of the Bounds on the Probability of Ruin for L\'{e}vy Processes with Light-tailed Jumps}

\usepackage{amsmath}
\usepackage{amsthm}
\usepackage{amssymb}
\usepackage{graphicx}
\usepackage{url}
\usepackage{xcolor}

\usepackage[margin=1.5in]{geometry}
\usepackage{lipsum,csquotes,setspace}
\usepackage{tikz}
\usetikzlibrary{arrows,calc}
\tikzset{
>=stealth',
help lines/.style={dashed, thick},
axis/.style={<->},
important line/.style={thick},
connection/.style={thick, dotted},
}
\makeatletter
\def\nobottom{%
\def\@texttop{\ifnum\c@page>0\vskip \z@ plus 3fil\relax\fi}
\def\@textbottom{\ifnum\c@page>0\vskip \z@ plus 2fil\relax\fi}}
\def\resetopandbottom{%
\def\@texttop{\ifnum\c@page>0\vskip 0pt plus .00006fil\relax\fi}
\def\@textbottom{\ifnum\c@page>0\vskip 0pt plus .00006fil\relax\fi}}
\makeatother
\usepackage{mathtools}
\usepackage{algorithm}
\usepackage[noend]{algpseudocode}
\usepackage{enumerate}
\usepackage{bbm}

\usepackage{epstopdf}

\newcommand{\expect}{\operatorname{\mathbb{E}}\expectarg}
\DeclarePairedDelimiterX{\expectarg}[1]{(}{)}{%
  \ifnum\currentgrouptype=16 \else\begingroup\fi
  \activatebar#1
  \ifnum\currentgrouptype=16 \else\endgroup\fi
}

\newcommand{\innermid}{\nonscript\;\delimsize\vert\nonscript\;}
\newcommand{\activatebar}{%
  \begingroup\lccode`\~=`\|
  \lowercase{\endgroup\let~}\innermid 
  \mathcode`|=\string"8000
}

\DeclarePairedDelimiterX{\expectargq}[1]{(}{)}{%
  \ifnum\currentgrouptype=16 \else\begingroup\fi
  \activatebarq#1
  \ifnum\currentgrouptype=16 \else\endgroup\fi
}

\newcommand{\innermidq}{\nonscript\;\delimsize\vert\nonscript\;}
\newcommand{\activatebarq}{%
  \begingroup\lccode`\~=`\|
  \lowercase{\endgroup\let~}\innermidq
  \mathcode`|=\string"8000
}

\newtheorem{defi}{Definition}
\newtheorem{prop}{Proposition}
\newtheorem{thm}{Theorem}
\newtheorem{rmk}{Remark}
\newtheorem{ass}{Assumption}
\newtheorem{coro}{Corollary}

\newcommand{\nN}{\mathbb{N}}
\newcommand{\rR}{\mathbb{R}}
\newcommand{\pP}{\mathbf{P}}

\newcommand{\eE}{\mathbf{E}}
\newcommand{\fF}{\mathbb{F}}

\begin{document}

\begin{abstract}
In this note, we study the ultimate ruin probabilities of a real-valued L\'{e}vy process $X$ with light-tailed negative jumps. It is well-known that, for such L\'{e}vy processes, the probability of ruin decreases as an exponential function with a rate given by the root of the Laplace exponent, when the initial value goes to infinity. Under the additional assumption that $X$ has integrable positive jumps, we show how a finer analysis of the Laplace exponent gives in fact a complete description of the bounds on the probability of ruin for this class of L\'{e}vy processes. This leads to the identification of a case that was not considered before. We apply the result to the Cram\'{e}r-Lundberg model perturbed by Brownian motion.
\vspace{3mm}

\end{abstract}

\maketitle

\noindent \subjclass{MSC 2010 subject classifications: }{60G51, 91B30}\\
\subjclass{JEL Classification : }{G220}

\section{Introduction and Main Result}

Ruin theory studies in particular the time of passage below $0$ of stochastic processes that represent the capital of an insurance company or a pension fund. In particular, it studies the probability that the process becomes negative on an infinite time horizon in function of the initial value of the process. The key result of Cram\'{e}r \cite{cramer1938} is that, in the case of the compound Poisson process with drift, this probability decreases as an exponential function with a rate given as a solution to the \emph{Lundberg equation}. It is well-know that, when the initial value goes to infinity, the result of Cram\'{e}r holds for more general (light-tailed) L\'{e}vy processes where the rate is given by the root of the Laplace exponent of the process, see Theorem XI.2.6 in \cite{asmussen2010}, and also \cite{bertoin1994}, \cite{kkm} and Section 7.2 in \cite{kyp2014}.

In this note, we show that a finer analysis of the Laplace exponent can lead to a complete description of the bounds on the ultimate probability of ruin. Our main contribution is to give a systematic description of all possible cases (Theorem 1), where the case when it has a root (Theorem 1, Case B) corresponds to the well-known Lundberg bound. This also leads to the identification of a case that is not considered in the literature (Theorem 1, Case D). We show that in this case the ruin probability also decreases at least as an exponential function and identify the rate of decay. Thus, Theorem 1 gives a method for obtaining exponential bounds and conditions for ruin with probability one for a large class of risk models. We illustrate this by applying the method to the Cram\'{e}r-Lundberg model perturbed by Brownian motion (Proposition 1).

When the L\'{e}vy process has jumps only on one side (i.e., it is \emph{spectrally one-sided}), the results contained in Chapter 8 of \cite{kyp2014} and the references therein give a precise description of the ultimate ruin probability in terms of the so-called \emph{scale functions}. However, these scale functions are in general not very explicit. In comparison, the method presented here is more elementary and less precise but works also in the case where there are two-sided jumps and is, in some cases, more explicit.

\subsection{L\'{e}vy Processes and Laplace Exponents}

In this section, we state some basic facts about L\'{e}vy processes and present the main assumptions for the rest of this paper.


Let $X = (X_t)_{t \geq 0}$ be a real-valued L\'{e}vy process on $(\Omega, \mathcal{F}, \fF = (\mathcal{F}_t)_{t \geq 0}, \pP)$ (in the sense of \cite{js2003}, Definition II.4.1, p.101) where the filtration $\fF$ is assumed to satisfy the usual conditions. It is well-known that the characteristic function of $X_t$ for each $t \geq 0$ is given by the L\'{e}vy-Khintchine formula:
$$\eE\left(e^{i\lambda X_t}\right) = e^{t\Phi(\lambda)}, \text{ for all } t \geq 0 \text{ and } \lambda \in \rR,$$
\noindent where 
\begin{equation*}
\Phi(\lambda) = ia\lambda - \frac{\sigma^2}{2}\lambda^2 + \int_{\rR}\left(e^{i\lambda x}-1-i\lambda x \mathbbm{1}_{\{|x| < 1\}}\right)\Pi(dx), \text{ for all } \lambda \in \rR,
\end{equation*}
\noindent for $a \in \rR$, $\sigma \geq 0$ and $\Pi$ a L\'{e}vy measure on $\rR$ satisfying $\Pi(\{0\}) = 0$ and
$$\int_{\rR}\left(x^2 \wedge 1\right) \Pi(dx) < +\infty.$$

\noindent The function $\Phi$ and the triplet $(a, \sigma^2, \Pi)$ are unique and are called the L\'{e}vy exponent and the characteristics (or L\'{e}vy triplet) of $X$ respectively. 

\begin{ass}
$X$ is integrable.
\end{ass}

The first assumption we use is integrability. We say that $X$ is integrable if $\eE(|X_1|) < +\infty$ and it can be shown (see e.g. \cite{sato1999}, Theorem 25.3, p.159) that this is equivalent to the condition
$$\int_{|x| \geq 1}|x| \Pi(dx) < +\infty.$$

\noindent Under assumption (I), we can rewrite the L\'{e}vy exponent of $X$ as
\begin{equation}
\Phi(\lambda) = i\delta\lambda - \frac{\sigma^2}{2}\lambda^2 + \int_{\rR}\left(e^{i\lambda x}-1-i\lambda x\right)\Pi(dx), \text{ for all } \lambda \in \rR,
\end{equation}
\noindent where 
$$\delta \triangleq \eE(X_1) = a + \int_{|x|\geq 1}x\Pi(dx).$$

\noindent Also, under assumption (I), the L\'{e}vy-It\^{o} decomposition of $X$ is
\begin{equation}\label{li_int}
X_t = \delta t + \sigma W_t + \int_{0}^t\int_{\rR}x \left(\mu^X-\nu^X\right)(ds, dx), \text{ for all } t \geq 0,
\end{equation}
\noindent where $\mu^X$ is the jump measure of $X$, $\nu^X(ds, dx) = ds\Pi(dx)$ is the compensator of the jump measure (see \cite{js2003}, Theorem I.1.8, p.66) and $(W_t)_{t \geq 0}$ is a standard Brownian motion.

\begin{ass}
$X$ has light-tailed negative jumps.
\end{ass}

The second assumption we will use is a condition on the tail behaviour of the negative jumps. Similar definitions to the one below can be found on p.338 in \cite{asmussen2010} and p.164-165 in \cite{sato1999}.

\begin{defi}
Let $X = (X_t)_{t \geq 0}$ be a real-valued L\'{e}vy process with characteristics $(a, \sigma^2, \Pi)$. Let 
$$\gamma_c \triangleq \sup\left\{\gamma \geq 0 : \int_{-\infty}^{-1}e^{-\gamma x}\Pi(dx) < +\infty\right\}.$$
\noindent We say that $X$ has light-tailed negative jumps if $\gamma_c > 0$. (Note that $\gamma_c$ can take the value $+\infty$.)
\end{defi}


Under Assumptions (I) and (II), it is possible to show that the L\'{e}vy exponent exists also for any $\lambda = i\gamma$, with $\gamma \in [0, \gamma_c)$. In fact, when $\gamma \in [0, \gamma_c)$,
$$\Phi(i\gamma) = -\delta \gamma + \frac{\sigma^2}{2}\gamma^2 + \int_{\rR}\left(e^{-\gamma x} - 1 + \gamma x\right)\Pi(dx),$$
\noindent and letting $I_- \triangleq \int_{\rR_-}\left|e^{-\gamma x}-1 + \gamma x\right|\Pi(dx)$, we obtain using the Taylor formula,
\begin{equation*}
\begin{split}
I_- & \leq \int_{-1}^0\left|e^{-\gamma x} - 1 + \gamma x\right|\Pi(dx) + \int_{-\infty}^{-1}e^{-\gamma x}\Pi(dx) \\
& \leq \frac{\gamma^2}{2}\int_{-1}^0x^2\Pi(dx) + \int_{-\infty}^{-1}e^{-\gamma x}\Pi(dx) < +\infty.
\end{split}
\end{equation*}
\noindent On the other hand, letting $I_+ \triangleq \int_{\rR_+}\left|e^{-\gamma x} -1 + \gamma x\right|\Pi(dx)$ and using the Taylor formula and the assumption of integrability,
\begin{equation*}
\begin{split}
I_+ & = \int_0^1\left|e^{-\gamma x} - 1 + \gamma x\right|\Pi(dx) + \int_1^{\infty}\left|e^{-\gamma x} - 1 + \gamma x\right|\Pi(dx) \\
& \leq \frac{\gamma^2}{2}\int_0^1x^2\Pi(dx) + \Pi\left([1, +\infty)\right) + \gamma \int_1^{\infty}|x|\Pi(dx) < +\infty.
\end{split}
\end{equation*}
\noindent Therefore, it is possible to define the Laplace exponent of $X$ as the function $\Psi$ given by
\begin{equation*}\label{eqpsi}
\Psi(\gamma) \triangleq \Phi(i\gamma) = -\delta \gamma + \frac{\sigma^2}{2}\gamma^2 + \int_{\rR}\left(e^{-\gamma x} - 1 + \gamma x\right) \Pi(dx), \text{ for all } \gamma \in [0, \gamma_c).\end{equation*}

\begin{rmk}
The Laplace exponent is always defined on $\rR_-$ and can, under Assumptions (I) and (II), be defined on $(-\infty, \gamma_c)$.
\end{rmk}

\noindent From the L\'{e}vy-Khintchine formula, we see that the Laplace transform of $X_t$ is then given by
$$\eE\left(e^{-\gamma X_t}\right) = e^{t\Psi(\gamma)}, \text{ for all } t \geq 0 \text{ and } \gamma \in (-\infty, \gamma_c).$$

\subsection{Main Result and Application}

Suppose that $X = (X_t)_{t \geq 0}$ is a real-valued L\'{e}vy process satisfying assumptions (I) and (II). Let $Y^u_t \triangleq u + X_t$, for $t \geq 0$ and $u \geq 0$. We define the ultimate ruin probability as 
$$\pP\left(\inf_{0 \leq t < +\infty}Y^u_t \leq 0\right) = \pP\left(\inf_{0 \leq t < +\infty}X_t \leq -u\right) = \pP\left(\sup_{0 \leq t < +\infty}(-X_t) \geq u\right).$$

\noindent This can also be written as $\pP\left(\tau(u) < +\infty\right)$ where $\tau(u) \triangleq \inf\{t \geq 0 : X_t \leq -u\}$ and $\tau(u) \triangleq +\infty$, if $X$ never goes below $-u$. We are now ready to give the main result. 

\begin{thm}\label{theorem}
Let $X = (X_t)_{t \geq 0}$ be a (non-zero) real-valued L\'{e}vy process satisfying Assumptions (I) and (II) and $\Psi : [0, \gamma_c) \to \rR$ be the Laplace exponent of $X$. Then, there are only four possible cases.

\begin{enumerate}[(A)]
\item If $\Psi(\gamma) > 0$, for all $\gamma \in (0, \gamma_c)$, then $\pP(\tau(u) < +\infty) = 1$, for all $u \geq 0$.
\item If there exists $\gamma_0 \in (0, \gamma_c)$ such that $\Psi(\gamma_0) = 0$, then $\pP(\tau(u) < +\infty) \leq e^{-\gamma_0 u}$, for all $u \geq 0$.
\item If $\gamma_c = +\infty$ and $\Psi(\gamma) < 0$, for all $\gamma \in (0, +\infty)$, then $\sigma^2 = 0$, $\Pi(\rR_{-}) = 0$, $\delta > 0$ and which means that $X$ is a subordinator. Therefore, $\pP(\tau(u) < +\infty) = 0$, for all $u \geq 0$.
\item If $\gamma_c < +\infty$ and $\Psi(\gamma) < 0$, for all $\gamma \in (0, \gamma_c)$, then $\pP\left(\tau(u) < +\infty\right) \leq e^{-\gamma_c u}$, for all $u \geq 0$.
\end{enumerate}
\end{thm}

Thus, Theorem \ref{theorem} exhausts all possible cases and allows one to classify the behaviour of the ruin probability in function of the behaviour of the Laplace exponent for a large class of risk models. To illustrate how to use Theorem \ref{theorem}, we apply it to the Cram\'{e}r-Lundberg model perturbed by Brownian motion. This model, which is sometimes also called \emph{perturbed risk process} and was studied first in \cite{gerber1970}, is given by
\begin{equation}\label{eq_mod2}
Y_t^u = u + pt + \sigma W_t - \sum_{n = 1}^{N_t}U_n, \text{ for all } t \geq 0,
\end{equation}
\noindent where $p > 0$, $\sigma > 0$, $N = (N_t)_{t \geq 0}$ is a standard Poisson process with rate $\beta$, $(W_t)_{t \geq 0}$ is a standard Brownian motion and $U = (U_n)_{n \in \nN}$ is a sequence of i.i.d. exponential random variables with rate $\alpha$. Additionally, it is assumed that the processes $N$, $W$ and the sequence $U$ are independent from each other.

Then, the following proposition gives the description of the ruin probabilities for this model. Note that in contrast to the case when $\sigma^2 = 0$, there are two possible regimes when the \emph{safety loading condition} $p > \frac{\beta}{\alpha}$ is satisfied. This shows how the uncertainty in premium payments affects the ruin probability. Also note that this result is very explicit as the behaviour of the ruin probability only depends on the value of the parameters and that it gives the complete description of the possible cases.

\begin{prop}
Let $X = (X_t)_{t \geq 0}$ be a real-valued L\'{e}vy process with L\'{e}vy triplet $\Pi(dx) = \beta \alpha e^{\alpha x}\mathbbm{1}_{\{x \leq 0\}}dx$, $\sigma^2 > 0$ and $a = p + \int_{|x| < 1}x\Pi(dx)$ for some $p, \alpha, \beta > 0$. Then, $Y^u_t = u + X_t$, with $u \geq 0$, corresponds to the perturbed risk process given by (\ref{eq_mod2}). Let $\Delta \triangleq (\sigma^2\alpha - 2p)^2 + 8\sigma^2\beta$ and $\gamma_- \triangleq \frac{\sigma^2 \alpha + 2p - \sqrt{\Delta}}{2 \sigma^2}$.
\begin{itemize}
\item If $p \leq \frac{\beta}{\alpha}$, then $\pP(\tau(u) < +\infty) = 1$, for all $u \geq 0$.
\item If $p > \frac{\beta}{\alpha}$ and $\gamma_- < \alpha$, then $\pP(\tau(u) < +\infty) \leq e^{-\gamma_- u}$.
\item If $p > \frac{\beta}{\alpha}$ and $\gamma_- \geq \alpha$, then $\pP(\tau(u) < +\infty) \leq e^{-\alpha u}$.
\end{itemize}
\end{prop}

\begin{proof}
We have $\gamma_c = \alpha$ and $\delta = p - \frac{\beta}{\alpha}$. So, by Theorem 1 (A), we have ruin with probability one when $p \leq \frac{\beta}{\alpha}$ and we assume in the following that $p > \frac{\beta}{\alpha}$. For $\gamma \in (0, \alpha)$, we obtain
\begin{equation*}
\begin{split}
\Psi(\gamma) & = -\delta \gamma + \frac{\sigma^2}{2}\gamma^2 + \beta \alpha \int_{-\infty}^{0}(e^{-\gamma x} - 1 + \gamma x)e^{\alpha x}dx \\
& = - p\gamma + \frac{\sigma^2}{2}\gamma^2 - \frac{\beta \alpha}{\gamma - \alpha} - \beta \\
& = \frac{\gamma\left(\sigma^2\gamma^2 - (\sigma^2\alpha + 2p)\gamma + 2(p\alpha - \beta)\right)}{2(\gamma - \alpha)} = -\frac{1}{2}A(\gamma)B(\gamma),
\end{split}
\end{equation*}

\noindent where $A(\gamma) \triangleq \frac{\gamma}{\alpha - \gamma}$ and $B(\gamma) \triangleq \sigma^2\gamma^2 - (\sigma^2\alpha + 2p)\gamma + 2(p\alpha - \beta)$. To see if $\Psi$ has an other root along $0$, we need to consider the solutions of $B(\gamma) = 0$. This is an equation of second order with determinant $\Delta$. As $\Delta > 0$, $B$ has two distinct roots $\gamma_+$ and $\gamma_-$, given by
$$\gamma_{\pm} = \frac{\sigma^2\alpha + 2p \pm \sqrt{\Delta}}{2\sigma^2}.$$

\begin{figure}[h]
\begin{tikzpicture}[scale=1]
    \coordinate (y) at (0,5);
    \coordinate (x) at (5,2);
    \draw[<-] (y) node[above] {$B(\gamma)$} -- (0,0);
    \draw[->] (-1,2) --  (x) node[right]{$\gamma$};

    \draw[important line] (0,4.5) .. controls (2.5,0) .. (5,4);
    \filldraw [black] 
     (3,2) circle (1pt) node[below left, black] {$\alpha$};
    \filldraw [black] 
     (3.7,2) circle (1pt) node[below right, black] {$\gamma_+$};
    \filldraw [black] 
     (1.465,2) circle (1pt) node[above right, black] {$\gamma_-$};

     \begin{scope}[xshift=7cm]
      \coordinate (y2) at (0,5);
      \coordinate (x2) at (5,2);
      \draw[<-] (y2) node[above] {$B(\gamma)$} -- (0,0);
      \draw[->] (-1, 2) --  (x2) node[right]{$\gamma$};
       \draw[important line] (0,4.5) .. controls (2.5,0) .. (5,4);
    \filldraw [black] 
     (1,2) circle (1pt) node[below left, black] {$\alpha$};
    \filldraw [black] 
     (3.7,2) circle (1pt) node[below right, black] {$\gamma_+$};
    \filldraw [black] 
     (1.465,2) circle (1pt) node[above right, black] {$\gamma_-$};
      \end{scope}
\end{tikzpicture}
\caption{Behaviour of $B$ when $\gamma_- < \alpha$ (left) and $\gamma_- \geq \alpha$  (right).}
\end{figure}
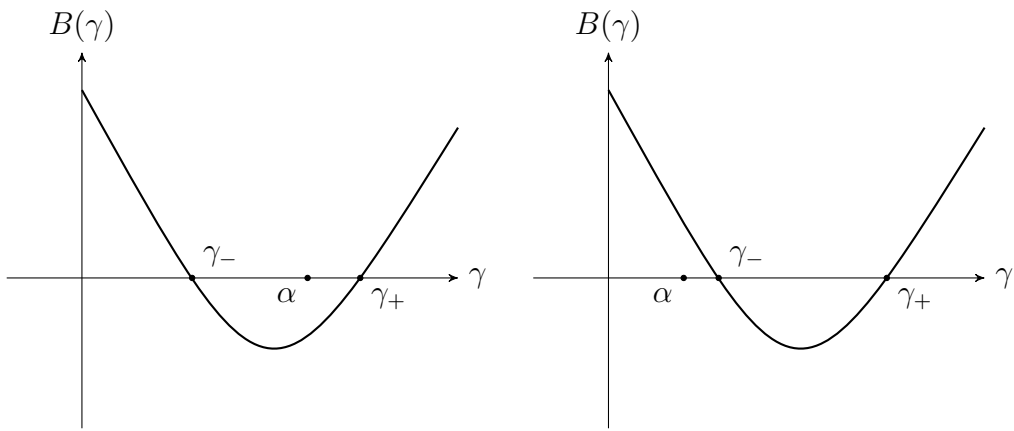

First note that $\gamma_- < \gamma_+$ and that $\gamma_+ \geq \alpha$ and $\gamma_- \geq 0$, because $(\sigma^2 \alpha + 2p)^2 \geq \Delta \geq (\sigma^2 \alpha - 2p)^2$. Additionally, note that $B''(\gamma) = 2\sigma^2 > 0$, so that $B$ is convex. Therefore, we only have two possible cases (see Figure 1) : either $\gamma_- < \alpha$ and then $\gamma_-$ is a root of $B$ and of $\Psi$, or $\gamma_- \geq \alpha$ and then $B(\gamma) > 0$ and $\Psi(\gamma) < 0$, for all $\gamma \in [0, \alpha)$. So, if $\gamma_- < \alpha$, then, by Theorem \ref{theorem} (B), we obtain $\pP(\tau(u) < +\infty) \leq e^{-\gamma_- u}$ and if $\gamma_- \geq \alpha$, then, by Theorem \ref{theorem} (D), we obtain $\pP(\tau(u) < +\infty) \leq e^{-\alpha u}$.
\end{proof}

\section{Proof of Theorem \ref{theorem}}

\subsection{Law of Large Numbers and Properties of the Laplace Exponent}

We start with the following well-known proposition and corollary (see Proposition IV.1.2, p.73 in \cite{asmussen2010} in the case of the compound Poisson process with drift, disscussion p.75 and Proposition 8 p.84 in \cite{bertoin1996}, Exercice 7.3 in \cite{kyp2014}, and Section 36 starting at p.245 in \cite{sato1999} in the general case) that give a strong law of large numbers and the tail behaviour for integrable L\'{e}vy processes. For completeness, we give an alternative proof which is not based on the random walk approximation.

\begin{prop}\label{p_aslimit}
Let $X = (X_t)_{t \geq 0}$ be real-valued L\'{e}vy process satisfying Assumption (I). Then, $\frac{X_t}{t} \overset{a.s.}{\to} \delta$, as $t \to +\infty$.
\end{prop}

\begin{proof}
Using the L\'{e}vy-It\^{o} decomposition (\ref{li_int}), we obtain
$$\frac{X_t}{t} = \delta + \sigma \frac{W_t}{t} + \frac{1}{t} \int_{0}^t\int_{\rR}x \left(\mu^X-\nu^X\right)(ds, dx), \text{ for all } t > 0.$$

\noindent But, $\frac{W_t}{t} \overset{a.s.}{\to} 0$. Now let $M \triangleq \int_{0}^.\int_{\rR}x \left(\mu^X-\nu^X\right)(ds, dx)$. We will show that $\frac{M_t}{t} \overset{a.s.}{\to} 0$. Note that 
\begin{equation*}
\begin{split}
M_t & = M^{(1)}_t + M^{(2)}_t + M^{(3)}_t \\
& \triangleq \int_{0}^t\int_{|x|< 1}x \left(\mu^X-\nu^X\right)(ds, dx) + \int_{0}^t\int_{|x|\geq 1}x \mu^X(ds, dx) \\
& - \int_{0}^t\int_{|x|\geq 1}x ds \Pi(dx).
\end{split}
\end{equation*}
\noindent Let's prove first that $\frac{M^{(1)}_t}{t} \overset{a.s.}{\to} 0$. By Theorem 9, p.142 in \cite{lipshir1989} it is enough to show that $\tilde{B}_{\infty} < +\infty$ a.s., where $\tilde{B}$ is the compensator of of the process $(B_t)_{t \geq 0}$ defined by
$$B_t = \sum_{0 \leq s < t}\frac{(\Delta M^{(1)}_s/(1+s))^2}{1 + |\Delta M^{(1)}_s/(1+s)|}, \text{ for all } t \geq 0,$$
\noindent where $\Delta M^{(1)}_s$ is the jump of $M^{(1)}$ at $s \geq 0$. But, by Theorem 1, p.176 in \cite{lipshir1989}, and using the fact that $\nu^X(\{s\},dx) = \lambda(\{s\})\Pi(dx) = 0$, because $\lambda$ is the Lebesgue measure, we obtain $\Delta M^{(1)}_s = \Delta X_s \mathbbm{1}_{\{|\Delta X_s| < 1\}}$. Next, note that
$$B_t = \sum_{0 \leq s < t}\frac{(\Delta X_s)^2 \mathbbm{1}_{\{|\Delta X_s| < 1\}}/(1+s)}{1 + s + |\Delta X_s \mathbbm{1}_{\{|\Delta X_s| < 1\}}|} = \int_0^t\int_{\rR}\frac{x^2 \mathbbm{1}_{\{|x| < 1\}}/(1+s)}{1 + s + |x \mathbbm{1}_{\{|x| < 1\}}|}\mu^X(ds, dx).$$

\noindent Therefore, $\tilde{B}$ satisfies
\begin{equation*}
\begin{split}
\tilde{B}_t = \int_0^t\int_{\rR}\frac{x^2 \mathbbm{1}_{\{|x| < 1\}}/(1+s)}{1 + s + |x \mathbbm{1}_{\{|x| < 1\}}|}ds\Pi(dx) \leq \left(\int_0^t\frac{1}{(1+s)^2}ds\right)\left(\int_{|x| < 1}x^2\Pi(dx)\right) \\
\leq \left(\int_0^{\infty}\frac{1}{(1+s)^2}ds\right)\left(\int_{|x| < 1}x^2\Pi(dx)\right) = \int_{|x| < 1}x^2\Pi(dx) < +\infty
\end{split}
\end{equation*}

\noindent for all $t \geq 0$, where the last integral is finite because $\Pi$ is a L\'{e}vy measure. So, $\tilde{B}_{\infty} < +\infty$ a.s. and, if $\Pi(|x| \geq 1) = 0$, we are finished. Therefore, without loss of generality, we suppose that $\Pi(|x| \geq 1) > 0$. Note that $\frac{M^{(3)}_t}{t} = -\int_{|x| \geq 1}x\Pi(dx)$, for all $t \geq 0$, so we need to show that $\frac{M^{(2)}_t}{t} \overset{a.s.}{\to} \int_{|x| \geq 1}x\Pi(dx)$.

It is well known that the jump measure $\mu^X$ is a Poisson random measure with intensity $\lambda \times \Pi$, where $\lambda$ is the Lebesgue measure. Then, by Lemma 2.8, p.46-47 in \cite{kyp2014}, $M^{(2)}$ can be represented as a compound Poisson process with rate $\Pi(|x| \geq 1)$ and jump distribution $\Pi(|x| \geq 1)^{-1}\Pi(dx)|_{\{|x| \geq 1\}}$ (where $\Pi(dx)|_{\{|x| \geq 1\}}$ is the restriction of the measure $\Pi$ to the set ${\{|x| \geq 1\}}$). More precisely,
$$M^{(2)}_t = \sum_{i = 1}^{N_t}Y_i, \text{ for all } t \geq 0,$$
\noindent where $(N_t)_{t \geq 0}$ is a Poisson process with rate $\Pi(|x| \geq 1)$ and $(Y_i)_{i \in \nN}$ is a sequence of i.i.d. random variables, which is independent from $N$ and with distribution $\Pi(|x| \geq 1)^{-1}\Pi(dx)|_{\{|x| \geq 1\}}$. Conditioning on $N_t$, using the strong law of large numbers and noting that $N_t \overset{a.s.}{\to} +\infty$, we obtain
$$\frac{M^{(2)}_t}{N_t} = \frac{1}{N_t}\sum_{i = 1}^{N_t}Y_i \overset{a.s.}{\to} \eE(Y_1) = \Pi(|x| \geq 1)^{-1}\int_{|x| \geq 1}x\Pi(dx).$$
\noindent Finally, using the fact that $\frac{N_t}{t} \overset{a.s.}{\to} \Pi(|x| \geq 1)$, we obtain 
$$\frac{M^{(2)}_t}{t} = \frac{N_t}{t}\frac{M^{(2)}_t}{N_t} \overset{a.s.}{\to} \int_{|x| \geq 1}x\Pi(dx).$$
\end{proof}

\begin{coro}\label{cor_plimit}
Let $X = (X_t)_{t \geq 0}$ be a (non-zero) real-valued L\'{e}vy process satisfying Assumption (I).
\begin{enumerate}
\item If $\delta > 0$, then $\lim_{t \to +\infty} X_t \overset{a.s}{=} +\infty$.
\item If $\delta < 0$, then $\lim_{t \to +\infty} X_t \overset{a.s}{=} -\infty$.
\item If $\delta = 0$, then $\liminf_{t \to +\infty} X_t \overset{a.s.}{=} -\infty$ and $\limsup_{t \to +\infty} X_t \overset{a.s.}{=} +\infty$.
\end{enumerate}
\end{coro}

\begin{proof}
The assertions 1 and 2 follow directly from Proposition \ref{p_aslimit}. For assertion 3, note that the condition $\delta = \eE(X_1) = 0$ implies, by Theorem 36.7, p.248 in \cite{sato1999}, that $X$ is recurrent. This means that we have neither $\lim_{t \to +\infty} X_t \overset{a.s}{=} +\infty$, nor $\lim_{t \to +\infty} X_t \overset{a.s}{=} -\infty$. Therefore, by Proposition 37.10, p.255 in \cite{sato1999}, $\liminf_{t \to +\infty} X_t \overset{a.s.}{=} -\infty$ and $\limsup_{t \to +\infty} X_t \overset{a.s.}{=} +\infty$.

%
\end{proof}

Next, the following proposition gives the basic properties of the Laplace exponent (see Lemma 26.4, p.169 in \cite{sato1999}).

\begin{prop}\label{prop_lapl}
Let $X = (X_t)_{t \geq 0}$ be a (non-zero) real-valued L\'{e}vy process satisfying Assumptions (I) and (II) and $\Psi : [0, \gamma_c) \to \rR$ the Laplace exponent of $X$. Then,
\begin{enumerate}
\item $\Psi$ is convex and starting from $0$ and
\item $\Psi$ is of class $C^{\infty}$ on $(0, \gamma_c)$ and its derivative $\Psi'$ is non-decreasing and given by
\begin{equation}\label{eqpsiprime}
\Psi'(\gamma) = -\delta + \sigma^2 \gamma + \int_{\rR}x\left(1-e^{-\gamma x}\right)\Pi(dx), \text{ for all } \gamma \in (0, \gamma_c).
\end{equation}
\end{enumerate}
\end{prop}

The convexity of the Laplace exponent then implies that there are only four possible cases which are illustrated in Figure 2 and reflect the possible cases for the behaviour of the ruin probability.

\begin{figure}[h]
\begin{center}
   \begin{tikzpicture}[scale=1]
    \coordinate (y) at (0,5);
    \coordinate (x) at (5,2);
    \draw[<-] (y) node[above] {$\Psi(\gamma)$} -- (0,0);
    \draw[->] (-1,2) --  (x) node[right]{$\gamma$};

    \draw[important line] (0,2) .. controls (1,1.2) and (4,3) .. (5,4) node[right] {$(B)$};
    \draw[important line] (0,2) .. controls (1,2) and (4,3) .. (5,4.5) node[right] {$(A)$};
    \draw[important line] (0,2) .. controls (1,1) and (3,0) .. (5,0) node[right] {$(C)$};
    \filldraw [black] 
     (1.73,2) circle (1pt) node[below right, black] {$\gamma_0$};

     \begin{scope}[xshift=7cm]
    \draw[<-] (0,5) node[above] {$\Psi(\gamma)$} -- (0,0);
    \draw[->] (-1,2) --  (5,2) node[right]{$\gamma$};
    \draw[dashed] (4,0) --  (4,5);

    \draw[important line] (0,2) .. controls (1,1.2) and (4,2) .. (4,4) node[right] {$(B)$};
    \draw[important line] (0,2) .. controls (1,2) and (3,2) .. (4,4.5) node[right] {$(A)$};
    \draw[important line] (0,2) .. controls (1,1) and (3,0) .. (4,0.5) node[right] {$(D)$};
    \filldraw [black] 
     (4,2) circle (1pt) node[below right, black] {$\gamma_c$};
     \filldraw [black] 
     (2.5,2) circle (1pt) node[below right, black] {$\gamma_0$};
      \end{scope}
\end{tikzpicture}
   \caption{Possible behaviours of the Laplace exponent $\Psi$ when $\gamma_c = +\infty$ (left) and $\gamma_c < +\infty$ (right).}
\end{center}
\end{figure}
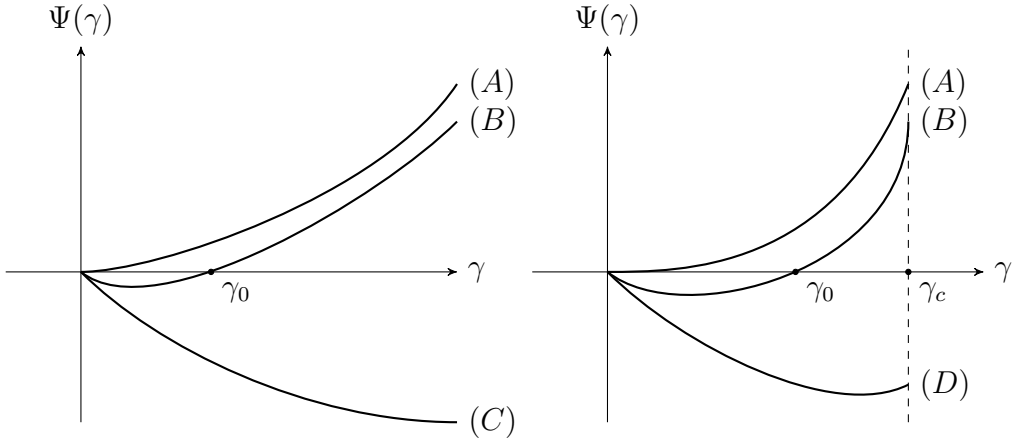

\subsection{The Martingale Method in Ruin Theory and the Proof}

In this final section, we recall the martingale method in ruin theory and apply it to prove Theorem \ref{theorem}. For the proof of the following well-known martingale method see e.g. Proposition II.3.1, p.29 in \cite{asmussen2010}.

\begin{prop}\label{thmexpasmp}
Let $X = (X_t)_{t \geq 0}$ be a real-valued L\'{e}vy process. Suppose that
\begin{enumerate}[(i)]
\item there exists $\gamma_0 > 0$, such that $(e^{-\gamma_0 X_t})_{t \geq 0}$ is a martingale,
\item $X_t \overset{a.s.}{\to} +\infty$ as $t \to +\infty$ on the set $\{\tau(u) = +\infty\}$. 
\end{enumerate}

\noindent Then, for all $u \geq 0$, $\pP\left(\tau(u) < +\infty\right) = C(u)e^{-\gamma_0 u} \leq e^{-\gamma_0 u}$, where
\begin{equation*}
C(u) \triangleq \frac{1}{\expect{e^{\gamma_0 \xi(u)} | \tau(u) < +\infty}},
\end{equation*}
\noindent and $\xi(u) \triangleq - u - X_{\tau(u)}$.
\end{prop}

%
%
%

\begin{rmk}
As noted in \cite{asmussen2010}, p.339, it is hard to obtain an explicit expression for $C(u)$. However, in some cases, it is possible to compute $C(u)$. For example, if $X$ has no negative jumps then $C(u) = 1$, and if the jumps are bounded or exponential, it is possible to compute the constant explicitly, see e.g. Section 6c in \cite{asmussen2010}. There are also asymptotic expressions for $C(u)$ as $u \to +\infty$, see e.g. Corollary XI.2.7 p.339 in \cite{asmussen2010} and Section 7.2. in \cite{kyp2014}. As we concentrate on the rate of decay of the probability of ruin in the general case, we will set $C(u) = 1$ and keep in mind that more precise results can be obtained for specific models or asymptotics.
\end{rmk}

The following proposition now gives a simple sufficient condition for (i) in Proposition \ref{thmexpasmp} in terms of the  Laplace exponent.

\begin{prop}\label{proppsiroot}
Let $X = (X_t)_{t \geq 0}$ be a real-valued L\'{e}vy process satisfying Assumptions (I) and (II) and $\Psi : [0, \gamma_c) \to \rR$ be the Laplace exponent of $X$. Suppose there exists $\gamma_0 \in (0, \gamma_c)$ such that $\Psi(\gamma_0) = 0$. Then, $(e^{-\gamma_0 X_t})_{t \geq 0}$ is a martingale.
\end{prop}

\begin{proof}
From the definition of $\gamma_c$, we have that $\eE(e^{-\gamma X_t}) < +\infty$ for all $t \geq 0$ and $\gamma \in [0, \gamma_c)$. Imitating the proof of Theorem II.1.2, p.23 in \cite{asmussen2010}, we find that the process $(e^{-\gamma X_t} - e^{t \Psi(\gamma)})_{t \geq 0}$ is a martingale for each $\gamma \in [0, \gamma_c)$. In particular, if there exists $\gamma_0 > 0$ such that $\Psi(\gamma_0) = 0$, then $(e^{-\gamma_0 X_t})_{t \geq 0}$ is a martingale.
\end{proof}

Putting everything together, we can now prove the main theorem. Note that case (B) can also be deduced with some work from Proposition XI.2.3 and Theorem XI.2.6 p.337-338 in \cite{asmussen2010} and that case (A) is generally implicitly excluded by the \textit{safety loading requirement} $\delta > 0$.

\begin{proof}[Proof of Theorem \ref{theorem}]
Note that from (\ref{eqpsiprime}) we obtain
$$\lim_{\gamma \to 0+}\Psi'(\gamma) = \lim_{\gamma \to 0+}\left(-\delta + \sigma^2 \gamma + \int_{\rR}x\left(1-e^{-\gamma x}\right)\Pi(dx)\right) = -\delta.$$

\noindent Therefore, from the study of the function $\Psi$, we see that $\delta \leq 0$ in case (A), and $\delta > 0$ in cases (B), (C) and (D).

\textit{Case (A).} Let $u \geq 0$. In case (A), we have $\delta \leq 0$. Suppose first that $\delta < 0$, then, by Corollary \ref{cor_plimit}, $X_t \overset{a.s.}{\to} -\infty$ as $t \to +\infty$. This immediately implies that 
$$\pP\left(\inf_{t \geq 0}X_t \leq -u\right) \geq \pP\left(\inf_{t \geq 0}X_t = -\infty \right) = 1.$$

\noindent If $\delta = 0$, then by Corollary \ref{cor_plimit}, $\pP\left(\liminf_{t \to +\infty} X_t \leq -u\right) = 1$. As $\left(\{\inf_{t \geq n}X_t \leq -u\}\right)_{n \in \nN}$ is a decreasing sequence of events, $\pP\left(\inf_{t \geq m}X_t \leq -u\right) \leq \pP\left(\inf_{t \geq 0}X_t \leq -u\right)$, for each $m \in \nN$ and
\begin{equation*}
\begin{split}
\pP\left(\inf_{t \geq 0}X_t \leq -u\right) & \geq \lim_{m \to \infty}\pP\left(\inf_{t \geq m}X_t \leq -u\right) = \lim_{m \to \infty}\pP\left(\bigcap_{n = 0}^m\left\{\inf_{t\geq n}X_t \leq -u\right\}\right)\\
& = \pP\left(\bigcap_{n \in \nN}\left\{\inf_{t\geq n}X_t \leq -u\right\}\right) = \pP\left(\liminf_{t \to \infty}X_t \leq -u\right) = 1.
\end{split}
\end{equation*}

\textit{Case (B).} We will show that (i) and (ii) of Proposition \ref{thmexpasmp} hold. Because (B) holds, by Proposition \ref{proppsiroot}, (i) is satisfied. Now note that in case (B) we have $\delta > 0$ and, by Corollary 1, that $X_t \overset{a.s.}{\to} +\infty$, as $t \to +\infty$. So (ii) is also satisfied.

\textit{Case (C).} Because (C) holds, we have $\Psi(\gamma) < 0$, $\lim_{\gamma \to 0+}\Psi'(\gamma) = -\delta < 0$. We also have $\Psi'(\gamma) < 0$, for all $\gamma > 0$. But, from (\ref{eqpsiprime}), we see that $\Psi'(\gamma) \leq 0$, for all $\gamma > 0$, if, and only if, 
$$\sigma^2 \gamma + \int_{\rR}x\left(1-e^{-\gamma x}\right)\Pi(dx) \leq \delta, \text{ for all } \gamma > 0.$$
\noindent If $\sigma^2 > 0$, the limit of the left-hand side when $\gamma \to +\infty$ goes to $+\infty$, so this immediately implies that $\sigma^2 = 0$. Now let $I \triangleq \int_{\rR}x\left(1-e^{-\gamma x}\right)\Pi(dx)$, and note that 
$$I = \int_{\rR_-}x\left(1-e^{-\gamma x}\right)\Pi(dx) + \int_{0}^{1}x\left(1-e^{-\gamma x}\right)\Pi(dx) + \int_{1}^{+\infty}x\left(1-e^{-\gamma x}\right)\Pi(dx).$$

\noindent Note that $x(1-e^{-\gamma x}) \leq x$, for all $x \geq 1$ and $\gamma > 0$. So, taking the limit as $\gamma \to +\infty$ and using the dominated convergence theorem on the integral over $(1, +\infty)$ with Assumption (I), we obtain
$$\lim_{\gamma \to +\infty}\int_{-\infty}^1x\left(1-e^{-\gamma x}\right)\Pi(dx) \leq \delta - \int_1^{+\infty}x\Pi(dx).$$

\noindent But, $x(1-e^{-\gamma x}) \geq \gamma x^2$, for all $x < 0$ and $\gamma > 0$. The above inequality, therefore implies
$$\lim_{\gamma \to +\infty}\left(\gamma \int_{\rR_{-}}x^2\Pi(dx) + \int_{0}^{1}x\left(1-e^{-\gamma x}\right)\Pi(dx)\right) < +\infty,$$
\noindent which implies that $\int_{\rR_{-}}x^2\Pi(dx) = 0$. Now note that the function $x \mapsto x^2$ is strictly positive on $\rR_-$ except in $0$. But, by definition of the L\'{e}vy measure $\Pi(\{0\}) = 0$, so $x \mapsto x^2$ is strictly positive $\Pi$-a.e. So, $\int_{\rR_{-}}x^2\Pi(dx) = 0$ if, and only if, $\Pi(\rR_-) = 0$.

\textit{Case (D).} Let $u \geq 0$. Fix $\epsilon \in (0, \gamma_c)$ and define
$$Z^{\epsilon}_t = \frac{\Psi(\gamma_c - \epsilon)}{\gamma_c - \epsilon}t + X_t, \text{ for all } t \geq 0.$$
\noindent Then, because (D) holds $\Psi(\gamma_c - \epsilon) < 0$, so that $Z^{\epsilon}_t \leq X_t$, for all $t \geq 0$, and 
$$\pP\left(\inf_{0 \leq t < +\infty}X_t \leq -u\right) \leq \pP\left(\inf_{0 \leq t < +\infty}Z^{\epsilon}_t \leq -u\right).$$
\noindent Note that the Laplace exponent $\Psi^{\epsilon}$ of $Z^{\epsilon}$ is defined for $\gamma \in [0, \gamma_c)$ and given by
\begin{equation}
\begin{split}
\Psi^{\epsilon}(\gamma) & = -\left(\frac{\Psi(\gamma_c - \epsilon)}{\gamma_c - \epsilon} + \delta\right)\gamma + \frac{\sigma^2}{2}\gamma^2 + \int_{\rR}\left(e^{-\gamma x} -1 + \gamma x\right)\Pi(x) \\
& = - \frac{\Psi(\gamma_c - \epsilon)}{\gamma_c - \epsilon}\gamma + \Psi(\gamma).
\end{split}
\end{equation}
\noindent Now, we will show that $Z^{\epsilon}$ satisfies (i) and (ii) of Proposition \ref{thmexpasmp}. Condition (i) is satisfied for $\gamma_0 = \gamma_c - \epsilon$, because $\Psi^{\epsilon}(\gamma_c - \epsilon) = 0$. For condition (ii), note that because $\Psi^{\epsilon}$ has a root and is convex, we have $\lim_{\gamma \to 0+}(\Psi^{\epsilon})'(\gamma) < 0$. Thus, by Corollary 1, we obtain that $Z^{\epsilon}_t \overset{a.s.}{\to} +\infty$, so that (ii) is also satisfied. Therefore, we obtain
$$\pP\left(\tau(u) < +\infty\right) \leq e^{-(\gamma_c-\epsilon) u}.$$
\noindent As this is true for each $\epsilon \in (0, \gamma_c)$, we can let $\epsilon \to 0+$ to finish the proof.
\end{proof}

\bibliographystyle{plain}
\bibliography{angers_light}

\end{document}